\documentclass[11pt]{amsart}
\usepackage[all]{xy}
\usepackage{a4wide}
\usepackage{amssymb}
\usepackage{amsthm}
\usepackage{amsmath}
\usepackage{amscd,enumitem}
\usepackage{verbatim}
\usepackage{float}
\usepackage{color}
\usepackage[all]{xy}
\usepackage{hyperref}
\usepackage{cleveref}
\textheight8.8in \textwidth6.7in \numberwithin{equation}{section}

\theoremstyle{plain}
\newtheorem{theorem}{Theorem}
\numberwithin{theorem}{section}
\newtheorem{corollary}[theorem]{Corollary}

\newtheorem{lemma}[theorem]{Lemma}
\newtheorem{proposition}[theorem]{Proposition}

\theoremstyle{definition}

\theoremstyle{remark}
\newtheorem*{remark}{Remark}

\newcommand{\Z}{\mathbb{Z}}
\newcommand{\N}{\mathbb{N}}
\newcommand{\C}{\mathbb{C}}

\renewcommand{\H}{\mathbb{H}}

\newcommand{\fraka}{\mathfrak a}

\newcommand{\HH}{\mathbb{H}}
\newcommand{\Eis}{\mathcal{E}}
\newcommand{\sm}{\text {\rm sm}}

\newcommand{\SO}{\operatorname{SO}}

\newcommand{\SL}{{\text {\rm SL}}}

\newcommand{\pihol}{\pi_{hol}}

\newcommand{\SLZ}{\SL_2(\Z)}

\begin{document}

\title[Mock modular Eisenstein series with Nebentypus]{Mock modular Eisenstein series with Nebentypus}

\author{Michael H. Mertens, Ken Ono, and Larry Rolen}

\dedicatory{In celebration of Bruce Berndt's 80th birthday}

\address{Max-Planck-Insitut f\"ur Mathematik Bonn, Vivatsgasse 7, 53111 Bonn, Germany}
\email{mhmertens@mpim-bonn.mpg.de}

\address{Department of Mathematics, University of Virginia,
Charlottesville, VA 22904} \email{ken.ono691@gmail.com}

\address{Department of Mathematics, 1420 Stevenson Center, Vanderbilt University,
Nashville, TN 37240}
\email{larry.rolen@vanderbilt.edu}

\thanks{The second author thanks the support of the NSF and the Asa Griggs Candler Fund and the Thomas Jefferson Fund.
}

\begin{abstract} 
By the theory of Eisenstein series,  generating functions of various divisor functions arise as
modular forms. It is natural to ask whether further divisor functions arise systematically
in the theory of mock modular forms.  We establish, using the method of Zagier and Zwegers on holomorphic projection,
that  this is indeed the case for certain (twisted) ``small divisors'' summatory functions
$\sigma_{\psi}^{\sm}(n)$. More precisely, in terms of the
weight 2 quasimodular Eisenstein series $E_2(\tau)$ and a generic Shimura
theta function $\theta_{\psi}(\tau)$, we show that there is a constant $\alpha_{\psi}$ for which
$$
\Eis^{+}_{\psi}(\tau):= \alpha_{\psi}\cdot\frac{E_2(\tau)}{\theta_{\psi}(\tau)}+
\frac{1}{\theta_{\psi}(\tau)} \sum_{n=1}^\infty \sigma^{\sm}_\psi(n)q^n 
$$
is a half integral weight (polar) mock modular form.
These  include generating functions for combinatorial objects such as 
the Andrews $spt$-function and the ``consecutive parts'' partition function.
Finally, in analogy with Serre's result that the weight $2$ Eisenstein series
is a $p$-adic modular form, we show that these forms  possess canonical congruences  with  modular forms.
\end{abstract}

\maketitle

\section{Introduction and statement of results}
In the theory of mock theta functions and its applications to combinatorics as developed by Andrews, Hickerson, Watson, and many others, various formulas for $q$-series representations have played an important role. For instance, the generating function $R(\zeta;q)$ for partitions organized by their ranks is given by:
\[
R(\zeta;q):=\sum_{\substack{n\geq0\\ m\in\Z}}N(m,n)\zeta^mq^n=\sum_{n\geq0}\frac{q^{n^2}}{(\zeta q;q)_n(\zeta^{-1}q;q)_n}=\frac{1-\zeta}{(q;q)_{\infty}}\sum_{n\in\Z}\frac{(-1)^nq^{\frac{n(3n+1)}{2}}}{1-\zeta q^n},
\]
where $N(m,n)$ is the number of partitions of $n$ of rank $m$ and $(a;q)_n:=\prod_{j=0}^{n-1}(1-aq^j)$ is the usual $q$-Pochhammer symbol.
The function $R(\zeta;q)$, when $\zeta$ is specialized to any root of unity times a certain rational power of $q$, is (up to a certain rational power of $q$) a mock modular form on a congruence subgroup; see 
\cite{BringmannOno,Kang,Zagier} for discussion and applications.
For instance, Ramanujan's original example of the third order mock theta function $f(q)$ is
\[
f(q):=R(-1:q)=\sum_{n\geq0}\frac{q^{n^2}}{(-q;q)_n^2},
\]
which shows that the Fourier coefficients of $f(q)$ enumerate the number of partitions of $n$ with even rank minus the number of partitions of $n$ with odd rank. 

Many  identities involving partition ranks are subtle and require delicate $q$-series techniques to prove. Thus, it is interesting to ask for further constructions of mock theta functions that will likely play a future role in the study of ranks and related partition functions. Here we take inspiration from the fact that many classical modular forms are generated by Eisenstein series 
which, thanks to their representations as Lambert series, have nice $q$-expansions in terms of divisor power sums. Such representations have long been a staple in the theory of modular forms and their applications.

To review one classical situation, we first recall some notation.
The $v$-th power divisor functions $\sigma_v(n):=\sum_{d\mid n}d^v$, for odd $v\geq 3$, arise naturally in the theory of modular forms.
Indeed, their generating functions are essentially the
classical Eisenstein series $G_k(\tau)$, defined for positive even integers $k$ by
$$
G_k(\tau):=\frac{1}{2}\zeta(1-k) +\sum_{n=1}^{\infty} \sigma_{k-1}(n)q^n, 
$$
where $q:=e^{2\pi i \tau}$. 
In particular, if $k\geq 4$, then $G_k(\tau)$ is a weight $k$ holomorphic modular form
for the full modular group $\SL_2(\Z)$.  

Hecke's well-known theory \cite{Hecke} of modular forms includes many further divisor functions. For example, for a primitive Dirichlet character $\chi$ modulo $N$,
consider the complementary divisor functions
$$
\sigma_{v,\chi}(n):=\sum_{d\mid n} \chi(d)d^v \ \ \ \ {\text {\rm and}}\ \ \ \ 
\sigma'_{v,\chi}(n):= \sum_{d\mid n} \chi(n/d)d^v.
$$
If $k>2$ and $(-1)^k=\chi(-1)$, then Hecke's work implies that the functions
\begin{displaymath}
G_{k,\chi}(\tau):=\frac{1}{2}L(1-k,\chi)+\sum_{n=1}^{\infty} \sigma_{k-1,\chi}(n)q^n \ \ \ \ {\text {\rm and}}\ \ \ \ 
\widehat{G}_{k,\chi}(\tau):=\sum_{n=1}^{\infty}\sigma'_{k-1,\chi}(n)q^n,
\end{displaymath}
where $L(s,\chi)$ denotes the standard Dirichlet $L$-function of $\chi$,  are holomorphic weight $k$ modular forms on $\Gamma_0(N)$  with Nebentypus $\chi$. 

The weight $k=2$ case is particularly interesting.
Hecke obtained a modification of $G_2(\tau)$ which establishes that  
\[E_2^*(\tau):=1-24\sum_{n=1}^\infty \sigma_1(n)q^n-\frac{3}{\pi \Im(\tau)}\] 
is a weight 2 harmonic Maa{\ss} form. Therefore, we may view
 $E_2(\tau):=1-24\sum_{n=1}^{\infty}\sigma_1(n)q^n$  as a ``mock modular form'', the holomorphic part of the nonholomorphic modular form $E_2^{*}(\tau)$ (see for instance \cite[Section 6.1]{AMSColloq} for more details).
In view of this example, it is natural to ask whether there are further divisor functions which arise naturally in the theory of mock modular forms. 
We show that this is indeed the case for the generating functions of certain twisted  ``small divisors''  summatory functions.  In particular, we shall associate these functions with
Shimura's weight 1/2 and 3/2 theta functions.

To make this precise, suppose that
$\psi \! \pmod {f_{\psi}}$ is a primitive Dirichlet character with conductor $f_\psi>1$.  We then define
$\lambda_{\psi}:=(1-\psi(-1))/2$, so that $\lambda_{\psi}=0$ (resp. $\lambda_{\psi}=1$) if $\psi$ is even (resp. odd).
The corresponding Shimura theta function
\cite{Shimura} of weight $\lambda_{\psi}+\frac{1}{2}$ on $\Gamma_0(4 f_{\psi}^2)$ with Nebentypus $\psi$ is defined as
 \begin{equation}
 \theta_{\psi}(\tau):=\begin{cases} \sum_{n=1}^\infty\psi(n)q^{n^2} \ \ \ \ \ &{\text {\rm if $\lambda_{\psi}=0$}},\\ \ \ \ \\
 \sum_{n=1}^{\infty}\psi(n)nq^{n^2} \ \ \ \ \ &{\text {\rm if $\lambda_{\psi}=1$.}}
 \end{cases}
 \end{equation}
It is a well-known fact that $\theta_\psi$ is a modular form, more precisely we have
$\theta_{\psi}(\tau)\in M_{\frac{1}{2}}(\Gamma_0(4 f_{\psi}^2),\psi)$ if $\lambda_\psi=0$ resp. $\theta_\psi\in S_{\frac{3}{2}}(\Gamma_0(4 f_{\psi}^2),\psi\chi_{-4})$ with $\chi_{-4}$ denoting the non-trivial character modulo $4$ if $\lambda_{\psi}=1$.

We define a twisted  ``small divisors'' summatory function for any non-trivial Dirichlet character $\psi$.
To ease notation, if $n$ is a positive integer, then we define its small divisors by
\begin{equation}
D_n := \{ d\mid n \ : \  1\leq d\leq n/d \ \ {\text {\rm and}}\ \ 
d\equiv n/d \pmod 2\}.
\end{equation}
Then the $\psi$-twisted small divisors summatory function is defined by
\begin{equation}\label{SigmaSmallDefn}
\sigma^{\sm}_{\psi}(n):=\sum_{d\in D_n} \psi\left(\frac{(n/d)^2-d^2}{4}\right)d.
\end{equation}
We let $\Eis^{+}_{\psi}(\tau)$ denote its generating function divided by $\theta_{\psi}(\tau),$ as in the
following expression
\begin{equation}\label{EisPlusDefn}
\Eis^{+}_{\psi}(\tau):= \alpha_\psi \cdot \frac{E_2(\tau)}{\theta_{\psi}(\tau)}+\frac{1}{\theta_{\psi}(\tau)}\cdot \sum_{n=1}^\infty \sigma^{\sm}_\psi(n)q^n,
\end{equation}
where $\alpha_{\psi}$ is a suitable, explicitly computable constant (see \Cref{Proofs}), depending only on $\psi$.

It turns out that $\Eis^{+}_{\psi}(\tau)$ is a polar mock modular form, the holomorphic part of a weight $\frac{3}{2}-\lambda_{\psi}$
{\it polar harmonic Maa{\ss}  form}.  Recall that polar harmonic Maa{\ss} forms have holomorphic parts which are permitted to have poles
in $\H$ (see Ch. 13 of \cite{AMSColloq}).
Its nonholomorphic part is given by the normalized period integral 
\begin{equation}
\Eis^{-}_{\psi}(\tau):= (-1)^{\lambda_\psi}\frac{(2\pi)^{\lambda_\psi-1/2}i}{8\Gamma(1/2+\lambda_{\psi})}\cdot \int_{-\overline\tau}^{i\infty} \frac{\theta_{\overline \psi}(z)}{(-i(z+\tau))^{3/2-\lambda_\psi}}dz.
\end{equation}
  
Using the constructive method of Zagier and Zwegers 
(for example, see \cite{ARZ,BKZ, IRR}), where harmonic Maa{\ss} forms are constructed via holomorphic projection, we obtain the following theorem.
 
\begin{theorem}\label{HMFGenFcn}
Assuming the notation above, we have that $\Eis_{\psi}(\tau):=\Eis_{\psi}^{+}(\tau) +\Eis_{\psi}^{-}(\tau)$
is a weight  $\frac{3}{2}-\lambda_{\psi}$ polar harmonic weak Maa{\ss}  form on $\Gamma_0(4f_{\psi}^2)$ with Nebentypus
$\overline{\psi}\cdot \chi_{-4}^{\lambda_{\psi}}$. 
\end{theorem}

\begin{remark}
The harmonic Maa{\ss} forms $\Eis_{\psi}(\tau)$ are, up to a constant factor, preimages of $\theta_{\psi}(\tau)$ under the Bruinier-Funke $\xi$-operator (see Section 5.2 of \cite{AMSColloq}). In other words, Shimura's theta functions
are the {\it shadows} of these canonical polar harmonic Maa{\ss} forms.
\end{remark}
\begin{remark}
It is clear from the definition that the ``mock Eisenstein series'' $\Eis_\psi^+(\tau)$ can only have poles in the zeros of $\theta_\psi(\tau)$. 

Since the space $M_2^!(\Gamma_0(4f_\psi))$ of weakly holomorphic modular forms of weight $2$ is infinite-dimensional and $\theta_\psi$ has only finitely many poles in a fundamental domain of $\Gamma_0(4f_\psi))$, it is possible to construct a form $G_\psi\in M_2^!(\Gamma_0(4f\psi))$ such that $G_\psi/\theta_\psi$ has the same principal part at each zero of $\theta_\psi$ as $\Eis_\psi$, wherefore the difference $\Eis_\psi-G_\psi/\theta_\psi$ is a harmonic Maa{\ss} form without any poles in the upper half-plane. 
\end{remark}

\begin{remark}
An elementary computation reveals that the $q$-series generating function of $\sigma^\sm_\psi(n)$
(i.e. without the $1/\theta_{\psi}(\tau)$ factor) can be given in terms of derivatives of Appell-Lerch sums as studied by Zwegers \cite{Zwegers,Zwegers2}. For a special case of this, see for instance \cite[Lemma 2]{Mertens14}. The forms obtained in Theorem~\ref{HMFGenFcn} have the pleasing property that both parts  depend on $\psi$ in a simple way which is analogous to the standard
theory of holomorphic modular forms with Nebentypus.
\end{remark}

A natural question arises as to which of the functions $\Eis_\psi^+(\tau)$ from Theorem~\ref{HMFGenFcn} are actually mock modular forms, that is, which don't actually have poles on the upper half-plane. For example,
this situation occurs whenever $\theta_{\psi}(\tau)$
 has no zeros in the upper half-plane. Due to the infinite product expansion of the Dedekind eta function $\eta(\tau):=q^{\frac{1}{24}}\prod_{n\geq1}(1-q^n)$, eta-quotients automatically have no zeros on the upper half-plane. In fact, Rouse and Webb showed in \cite{RW} that the {\it only} modular forms on $\Gamma_0(N)$ with integer Fourier coefficients and no zeros on the upper half-plane are eta quotients. 
The classification of such theta functions was completed by Mersmann \cite{Mers} and Lemke Oliver in \cite{LO}.

\begin{theorem}\label{EtaQuotientList}
The only nontrivial primitive characters $\psi$ where $\theta_\psi(\tau)$ is an eta-quotient (in which case $\Eis_\psi^+(\tau)$ is a mock modular form), are when $\psi$ is in the set of Kronecker characters $\chi_D:=\left(\frac{D}{\cdot}\right)$ where $D$ is in the following set of fundamental discriminants:
\[
\left\{
-8,-4,-3,2,12,24
\right\}
.
\]
\end{theorem}

Several of
the mock Eisenstein series in Theorem~\ref{EtaQuotientList} occur naturally in the theory of partitions.
Here we point two particularly well known examples.  We let $spt(n)$ denote the Andrews {\it smallest parts function} \cite{AndrewsSPT}, which counts the number of smallest parts among partitions of $n$, whose mock modular properties were first exhibited and exploited by Bringmann \cite{Bringmann} and by Folsom and the second author in \cite{FolsomOno}.
Furthermore, we will denote by $cpt(n)$ the total number of parts in all partitions of $n$ into {\it consecutive integers}. Connections between odd divisors of integers and partitions into consecutive parts have also been noted previously (see works of Hirschhorn and Hirschhorn \cite{HH} as well as Sylvester and Franklin \cite{Sylv}). 

We summarize these combinatorial examples which arise out of Theorem~\ref{EtaQuotientList} in the following. 
\begin{corollary}\label{Examples} The following are true:
\begin{enumerate}\item[i)]
The functions 
\[\Eis_{\chi_{12}}^+(\tau)=-2q^{-1}\sum_{n\geq1}spt(n)q^{24n}+\frac{1}{6\theta_{\chi_{12}}(\tau)}E_2(24\tau)\]
and 
\[\Eis_{\chi_{24}}^+(\tau)=-2q^{-1}\sum_{n\geq1}(-1)^nspt(n)q^{24n}-\frac{1}{24\theta_{\chi_{24}}(\tau)}E_2(24\tau)\]
are mock modular forms of weight $3/2$.
\item[ii)]
The function
\[
\theta_{\chi_2}(\tau)\cdot\Eis_{\chi_2}^+(\tau)=2\sum_{n\geq1}(-1)^ncpt(n)q^{8n}+\frac 1{16}E_2(8\tau)
\]
is a  mixed mock modular form, i.e. the product of a mock modular form and a modular form (see for instance \cite[Section 13.2]{AMSColloq}  of weight $2$.
\end{enumerate}
\end{corollary}
Note that in slight abuse of notation, we have not inserted $E_2(\tau)$ in the above corollary, but rather $E_2(24\tau)$ resp. $E_2(8\tau)$ because thus the coefficients remain supported on the same arithmetic progression as indicated in the given sums and the transformation behavior of the completed function can be improved. For instance, the function $\theta_{\chi_{12}}(\tau/24)\Eis_{\chi_{12}}(\tau/24)$ turns out to be modular of weight $2$ for the full modular group $\SLZ$.

\begin{remark} There are other prominent examples which have appeared in the literature outside of
combinatorial number theory. For instance, in work related to quantum black holes, we find 
\[
\theta_{\chi_{-4}}(\tau)\cdot\Eis_{\chi_{-4}}^+\left(\frac{\tau}{8}\right)=\sum_{\substack{r>s>0\\ r-s\equiv1\pmod2}}(-1)^{r}sq^{\frac{rs}2}=:F_2^{(2)}(\tau),
\]
where $F_2^{(2)}(\tau)$ is the series which Dabholkar, Murthy, and Zagier computed using holomorphic projection (cf. \cite[Section~7.1, Example 4]{DMZ}). 
They use this form to construct the mock modular form
\[
h^{(2)}(\tau):=\frac{24F_2^{(2)}(\tau)-E_2(\tau)}{\eta(\tau)^3},
\]
which has shadow $\eta^3(\tau)$, and was first described in works of Eguchi-Taormina \cite{56} and Eguchi-Ooguri-Taormina-Yang \cite{53} in relation to the elliptic genus. It has since played an important role in Mathieu Moonshine \cite{24,52} and in relation to $\SO(3)$ Donaldson invariants of $\mathbb CP^2$ \cite{MO}.
\end{remark}

Although the Eisenstein series $E_2(\tau)$ is not modular, Serre \cite{Serre} proved that it is a $p$-adic modular form
for every prime $p$. Namely, for every prime power $p^t$ he identified  $E_2(\tau)\pmod{p^t}$ as the reduction
of a genuine holomorphic modular form. Moreover, he found that such sequences of modular forms, as $t\rightarrow +\infty$,
have weights that $p$-adically converge to 2. 

In view of Theorem~\ref{HMFGenFcn}, it is natural to ask whether the mock modular Eisentstein series
$\Eis_{\psi}^{+}(\tau)$ possess similar $p$-adic properties. In analogy with Serre's construction, we indeed show
that these mock modular Eisenstein series satisfy analogous congruences with genuine meromorphic modular forms modulo arbitrary prime powers.
Under repeated iteration of the $U$-operator
\begin{equation}
\left(\sum_{n\gg -\infty} a(n)q^n\right ) \  \Big| \ U(p) := \sum_{n\gg -\infty} a(pn)q^n,
\end{equation}
we obtain the following result.

\bigskip
\begin{theorem}\label{padicTheorem}
Let $p>3$ be a prime and $a,b\in\N$. Then there is a weight 2 meromorphic modular form $F_{\psi}^{(a,b)}(\tau)$ for the group $\Gamma_0(4f_\psi^2p^m)$, with $m=\max(2a-b,1-2v_p(f_\psi))$, with trivial Nebentypus such that we have the congruence
\[\left(\theta_\psi(p^{2a}\tau)\Eis_\psi^+(\tau)\right)|U(p^b)\equiv F_\psi^{(a,b)}(\tau)\pmod{p^{\min(a,b)}}.\]
\end{theorem}

\begin{remark}
In the special case of $spt(n)$, these $p$-adic properties 
have been studied previously.
Andrews found \cite{AndrewsSPT} that the $spt$-function satisfies the following Ramanujan-type congruences:
\[
spt(5n+4)\equiv0\pmod5,\quad\quad spt(7n+5)\equiv0\pmod7,\quad\quad spt(13n+6)\equiv0\pmod{13},
\]
and the second author and Folsom in \cite{FolsomOno}  obtained many further congruences, see also \cite{OnoPNAS}.
Garvan \cite{Garvanspt5713} and later Ahlgren and Kim \cite{AK}
used the modularity properties to study the context of Andrews' congruences. They showed that Andrews' congruences are ``explained'' by the triviality of the space of cusp forms of level one on and weights $5+1,7+1,13+1$, and surmised the general conjecture that no further congruences of this type exist for primes $\ell\geq11$.
As for prime power congruences, results close to $p$-adic modularity along the lines of Theorem~\ref{padicTheorem} have been obtained in \cite{REUPaper}. 
\end{remark}

The paper is organized as follows. In Subsection~\ref{HolProj} we will review the theory of holomorphic projection and establish some preliminary results, which will will exploit in Subsection~\ref{Proofs} to prove our main results on the construction of the polar mock modular Eisenstein series. We then derive the examples, including $spt(n)$ and $cpt(n)$, where we obtain mock modular forms in Section~\ref{ExSection}. We conclude in Section~\ref{PadicProofs} with the proof of our $p$-adic congruence results. 

\section*{Acknowledgments}
The authors thank Scott Ahlgren for noticing an error in an earlier version of this manuscript. We also thank the anonymous referee for various suggestions that helped improving the manuscript. 

\section{Modularity properties of new divisor function generating functions}

Here we prove Theorem~\ref{HMFGenFcn}. To this end we first recall the important work of Zwegers,
which we then manipulate to obtain the desired generating functions.

\subsection{Holomorphic projection}\label{HolProj}

Holomorphic projection is now an important tool in the study of non-holomorphic modular forms, in particular harmonic Maa{\ss} forms. It was first introduced by Sturm \cite{Sturm} and further developed by Gross and Zagier in their seminal work on derivatives of $L$-functions \cite{GrZ86}. For some applications in the context of mock modular forms and harmonic Maa{\ss} forms, the reader may consult for instance \cite{AA,ARZ,BKZ,IRR,Mert14,MO14}, among several others. For background on the subject of harmonic Maa{\ss} forms, we refer the reader to \cite{AMSColloq}.

Let us first recall the definition of holomorphic projection and explicit expressions for holomorphic projection of the product of a harmonic Maa{\ss} form with a holomorphic modular form. These can be found e.g. in \cite{IRR}, where a vector-valued version of the same question was considered, or in \cite{Mert14}, where Rankin-Cohen brackets of such functions were examined.

The usual holomorphic projection operator is defined for a (not necessarily holomorphic) modular form $\Phi$, growing at most polynomially approaching the cusps, with Fourier expansion
\begin{equation}\label{eqPhi}
\Phi(\tau)=\sum_{n\in\Z} c(n,y)q^n,\qquad y:=\Im(\tau),
\end{equation}
of weight $k\geq 2$ for some subgroup $\Gamma\leq \SLZ$  as
\[\pihol(\Phi)(\tau):=c_0+\sum_{n=1}^\infty c_nq^n\]
with
\begin{equation}\label{eqpihol}
c_n:=\frac{(4\pi n)^{k-1}}{\Gamma(k-1)}\int_0^\infty c(n,y)e^{-4\pi ny}y^{k-2}dy,\qquad (n\geq 1),
\end{equation}
provided that the integral converges, and where $c_0$ is an appropriate constant term. This operator maps into the space of holomorphic modular forms for $k>2$ and to the space of holomorphic quasimodular forms for $k=2$. It also behaves like a projection in that it acts as the identity on holomorphic functions, and it preserves the Petersson inner product, i.e. for every cusp form $g\in S_k(\Gamma)$ we have 
\[\langle \Phi,g\rangle=\langle \pihol(\Phi),g\rangle.\]
This is in fact the origin of the operator's definition.

We mention for the sake of completeness that in many applications, the growth conditions on the function $\Phi$ are too restrictive. If instead of moderate growth towards the cusps, one requires that at each cusp $\fraka$ there is a polynomial $H_\fraka\in\C[X]$ such that
\[\Phi-H_\fraka(q^{-1/h}),\]
where $h$ denotes the width of the cusp $\fraka$, has moderate growth approaching $\fraka$, and requires that the coefficients of $\Phi$ satisfy $c(n,y)=O(y^{2-k})$ as $y\to 0$ uniformly for all $n>0$, one obtains a \emph{regularized} holomorphic projection operator
\[\pihol^{reg}(\Phi)(\tau):=H_{i\infty}(q^{-1})+\sum_{n=1}^\infty \tilde c_nq^n\]
with
$$
\tilde c_n:=\frac{(4\pi n)^{k-1}}{\Gamma(k-1)}\lim_{s\to 0}\int_0^\infty c(n,y)e^{-4\pi ny}y^{k-2-s}dy,\qquad (n\geq 1).
$$
For further details, the reader may consult for instance \cite{MO14}. 

It is a basic fact (see for instance \cite[Section~4.2]{AMSColloq}) that a harmonic Maa{\ss} form $f\in H_k(\Gamma)$ of weight $k\neq 1$ naturally splits into a holomorphic and a non-holomorphic part,
\[f(\tau)=f^+(\tau)+f^-(\tau)\]
with
\begin{equation}\label{eqsplit}
f^+(\tau)=\sum_{n\gg -\infty} c_f^+(n)q^n\qquad\text{and}\qquad f^-(\tau)=\frac{(4\pi y)^{1-k}}{1-k}\overline{c_f^-(0)}+\sum_{n=1}^\infty \overline{c_f^-(n)}n^{k-1}\Gamma(1-k;4\pi ny)q^{-n}
\end{equation}
With this, one obtains the following result which is a special case of \cite[Theorem 4.6]{Mert14}. 
\begin{proposition}\label{propholproj}
Let $f\in H_k(\Gamma)$ be a harmonic Maa{\ss} form of weight $k\in\frac 12\Z\setminus\{1\}$ with an expansion as in \eqref{eqsplit} with $c_f^-(0)=0$  and $g\in M_{2-k}(\Gamma)$ with Fourier expansion
\[g(\tau)=\sum_{n=1}^\infty b(n)q^n.\]
Then the regularized holomorphic projection $\pihol^{reg}(f\cdot g)$ is well-defined and we have
\[\pihol^{reg}(f\cdot g)(\tau)=f^+(\tau)g(\tau)+\sum_{r=1}^\infty c(r)q^r\]
where
\[c(r)=\frac{\Gamma(2-k)}{k-1}\sum_{m=1}^\infty \overline{c_f^-(m)}b(m+r)\left[(r+m)^{k-1}-m^{k-1}\right].\] 
\end{proposition}
It is sometimes convenient to define holomorphic projection for merely translation-invariant rather than modular functions  as done for instance in \cite[Section 3]{BKZ}. The results there will be the key to the proof of Theorem~\ref{HMFGenFcn} and are summarized in the following lemma (cf. \cite[Lemmas 3.1--3.3]{BKZ}).
\begin{lemma}\label{lempihol}
Let $\Phi:\HH\to\C$ be a translation invariant function with Fourier expansion as in \eqref{eqPhi} and $\kappa>1$. Then the following are true:
\begin{enumerate}
\item We have that 
\begin{gather}\label{eqpiholint}
\pihol f(\tau)=\frac{(\kappa-1)(2i)^{\kappa}}{4\pi}\int_{\HH} \frac{f(x+iy)y^{\kappa}}{(\tau-x+iy)^{\kappa}}\frac{dx dy}{y^2},
\end{gather}
provided that both expressions are well-defined.
\item If $f:\HH\to\C$ is holomorphic and translation invariant so that the right-hand side of \eqref{eqpiholint} converges absolutely, we have for either of the possible expressions of $\pihol$ that $\pihol f=f$.
\item If the right-hand side of \eqref{eqpiholint} converges absolutely, we have for all $\gamma\in\SLZ$ and $\tau\in\HH$ that
$$\pihol (f|_\kappa \gamma)(\tau)=(\pihol f)|_\kappa\gamma(\tau).$$
\end{enumerate}
\end{lemma}
\begin{remark}
We point out that for instance \Cref{propholproj} holds for translation-invariant functions, provided that their Fourier expansions have the correct form.
\end{remark}

From this, one can derive the following result from \cite[Proposition 3.5]{BKZ}.
\begin{proposition}\label{propmod}
Let $f\colon \HH\to\C$ be a translation invariant function such that $|f(\tau)|\Im(\tau)^r$ is bounded on $\HH$ for some $r>0$. If the weight $k$ holomorphic projection of $f$ vanishes identically for some $k>r+1$ and $\xi_k f$ is modular of weight $2-k$ for some subgroup $\Gamma\leq \SLZ$, then $f$ is modular of weight $k$ for $\Gamma$.
\end{proposition}

\subsection{Proof of Theorem~\ref{HMFGenFcn}}\label{Proofs}

\bigskip

\begin{proof}[Proof of Theorem~\ref{HMFGenFcn}]
It is clear from the definition that $\Eis_\psi$ is annihilated by the weight $1/2+\lambda_\psi$ hyperbolic Laplacian away from the singularities, which lie precisely in the zeros of $\theta_{\psi}$. Since the Fourier coefficients of $\theta_\psi\Eis_\psi^+$ at $\infty$ grow at most polynomially, wherefore the growth of the function towards any cusp can be at most polynomially and thus $\Eis_\psi$ itself can only grow at most linearly exponentially towards any cusp. It follows from the known transformation properties of $\theta_\psi$ under the full modular group that $\Eis_\psi^-$ also grows at most polynomially towards the cusps, so that $\Eis_\psi$ has the correct growth properties towards the cusps for a harmonic Maa{\ss} form. We are therefore left with showing modularity. The expression for $\pihol(\theta_\psi\Eis_\psi)$ in \Cref{lempihol} (1) is easily seen to be well-defined, because of the polynomial growth towards the cusps, as discussed above. A standard computation further reveals that
\[\Eis_\psi^-=
 \frac{1/2-\lambda_\psi}{2\Gamma(1/2+\lambda_{\psi})}\sum_{n=1}^\infty \overline{\psi(n)}n^{1-2\lambda_\psi}\Gamma(\lambda_\psi-1/2;4\pi n^2y)q^{-n^2}.\]
Under the assumption that the constant $\alpha_\psi$ in the definition of $
\Eis_\psi^+$ is zero,  we obtain from \Cref{propholproj} (see also the remark following \Cref{lempihol})  that
\[\pihol(\theta_\psi\Eis_\psi)=\sum_{n=1}^\infty \sigma^{\sm}_\psi(n)q^n-\sum_{m>r\geq 1}\psi(mr)(m-r)q^{m^2-r^2},\]
which is easily seen to be identically zero. Furthermore, we have that
\[\xi_{2}\left(\theta_\psi\Eis_\psi\right)(\tau)=C y^{\lambda_\psi+1/2}\vert\theta_\psi(\tau)\vert^2\]
for some constant $C$ is clearly modular of weight $0$ for the group $\Gamma_0(4f_\psi^2)$ with Nebentypus $\chi_{-4}^{\lambda_\psi}$. However, \Cref{propmod} does not immediately apply since the required growth conditions are not satisfied. However, by writing the generating function $\sum_{n=1}^\infty\sigma^\sm_\psi(n)q^n$ in terms of derivatives of Appell-Lerch sums (cf. the third remark following \Cref{HMFGenFcn} and \cite[Chapter IV]{MertPhD}), and using the known modular properties of these functions (see for instance \cite{Zwegers2}), one can compute the asymptotic of this function as $\tau$ approaches any rational number.  We find that there is a linear combination $F(\tau)$ of (modular) Eisenstein series for $\Gamma_0(4f_\psi^2)$ of weight $2$ as well as the quasimodular Eisenstein series $E_2$, such that the function
$$\sum_{n=1}^\infty\sigma^\sm_\psi(n)q^n+\theta_\psi(\tau)\Eis^-(\tau)+F(\tau)+\alpha_\psi E_2(\tau)$$
behaves like a cusp form when $\tau$ approaches a rational number. We leave the details to the reader. Therefore this function satisfies the hypotheses of \Cref{propmod}. Adding modular Eisenstein series does not affect modularity, so we may deduce that 
$$\sum_{n=1}^\infty\sigma^\sm_\psi(n)q^n+\theta_\psi(\tau)\Eis^-(\tau)+\alpha_\psi E_2(\tau)$$
is modular. This concludes the proof.
\end{proof}
\begin{remark}
A very similar complication for holomorphic projection in weight $2$ is dealt with in an analogous way in the seminal work of Gross and Zagier \cite[Proposition 6.2]{GrZ86}, and has continued to arise in subsequent papers involving holomorphic projection. It would be very interesting to find
an exact (or more conceptual) formula for $\alpha_{\psi}$ in all of these situations, rather than having to rely on the solvability of a system of linear equations in modular forms.
\end{remark}
\begin{remark}
As an alternative to the proof given above, one might also define a modified regularized holomorphic projection operator for translation invariant functions by the constant term in the Laurent expansion around $s=0$ of
$$\frac{(\kappa-1)(2i)^{\kappa+2s}}{4\pi}\lim_{T\to\infty} \int_{\HH_T} \frac{f(x+iy)y^{k+2s}}{(\tau-x+iy)^{\kappa+2s}}\frac{dx dy}{y^2}.$$
This however does not in general produce holomorphic functions, but might add powers of $\Im(\tau)$, just like when one applies Hecke's trick to find the transformation properties of $E_2$. Other than that, \Cref{lempihol} essentially still holds. The power of $\Im(\tau)$ introduced by the regularization may be replaced by a multiple of $E_2$, again verifying the claim.
\end{remark}
\begin{remark}
The argument that enables us to use holomorphic projection in the above setup using Appell-Lerch sums would also suffice to establish modularity, without relying on holomorphic projection at all. 

It seems conceivable that one can derive asymptotics for Appell-Lerch sums near cusps more directly. For example a straightforward modification of \cite[Lemma 4.3]{BringmannRolen} shows that derivatives of Appell-Lerch sums have the right order of magnitude approaching cusps of $\Gamma_0(4N)$, so a more careful analysis should yield the correct asymptotic expansion.
\end{remark}

\section{Proof of Theorem~\ref{EtaQuotientList} and Corollary~\ref{Examples}}\label{ExSection} \ \ \ \\

The proof of Theorem~\ref{EtaQuotientList} follows directly from the characterized list of theta functions which are also eta quotients in \cite{LO}, as per the discussion preceding the theorem.

\begin{proof}[Proof of Corollary~\ref{Examples}]
Firstly, note that in part i) , the character $\chi_{24}$ is $\left(\frac2n\right)$ times the first character $\chi_{12}$. Thus,  switching from the first example to the second twists the theta function in the denominator of \eqref{EisPlusDefn}  by $\left(\frac2n\right)$, and a similar check can be performed on the character in the sum in \eqref{SigmaSmallDefn}. Thus, it is readily checkable that the extra twist by $(-1)^n$ occurs in this case.

It remains to check that the $spt$ and $cpt$ functions arise out of these constructions. 
For both, the mock modularity properties are already known, at least implicitly in the literature. For instance, the mock modularity of the generating function of $spt(n)$ is described in \cite{FolsomOno}, and our Theorem~\ref{EtaQuotientList} establishes the mock modularity of our series $\Eis_{\chi_{12}}^+(\tau)$.
Similarly, for $cpt(n)$, the 
generating function can be shown elementary means to have an Appell-Lerch-type generating function 
\[\sum_{n\geq1}cpt(n)q^n=\sum_{n\geq1}\frac{nq^{\frac{n(n+1)}2}}{1-q^n}\] 
(see the proof of Joerg Arndt in the comments on the OEIS sequence A204217). This is readily expressible as a theta function times a Taylor coefficient in an elliptic variable of Zwegers' $\mu$-function \cite{Zwegers}, which by the general theory there is a  weight $3/2$ mock modular form. Normalizing by dividing by the theta function, the identity for $cpt(n)$ is thus reduced to checking that two weight $3/2$ mock modular forms in known spaces are equal. Thus, to prove the claims, one first needs to check that they live in the same space of mock modular forms, that is, have the same weight, level, Nebentypus, shadow, and that their principal parts at the cusps agree.

Let us focus on the case of $spt$, it works in much the same way for $cpt$. It has been established in \cite{Bringmann} that 
$$F(\tau):=q^{-1}\sum_{n=1}^\infty spt(n)q^{24n}-\frac{1}{12}\frac{E_2(24\tau)}{\eta(24\tau)}-\frac{i}{4\pi\sqrt{2}}\int_{-\overline\tau}^\infty \frac{\eta(24z)}{(i(z+\tau))^{3/2}}dz$$
is a harmonic Maa{\ss} form for $\Gamma_0(576)$ with Nebentypus $\chi_{12}=\left(\frac{12}{\bullet}\right)$. By direct comparison, using the famous fact that $\theta_{\chi_{12}}(\tau)=\eta(24\tau)$, we therefore see from \Cref{EtaQuotientList} that $\Eis_{\chi_{12}}-2F$ is a weakly holomorphic modular form of weight $3/2$ for $\Gamma_0(576)$ with Nebentypus $\chi_{12}$. After multiplying it by a suitable cusp form (as it turns out $\eta(24\tau)$ suffices), we are left with checking that a holomorphic modular form is zero, which can be done checking only finitely many coefficients. 
\end{proof}

\section{Proof of Theorem~\ref{padicTheorem}}\label{PadicProofs}

We now turn to the behavior of the generating functions modulo powers of primes.
\begin{proof}[Proof of \Cref{padicTheorem}]
As remarked after the statement of \Cref{HMFGenFcn}, there is a weakly holomorphic modular form $G_\psi\in M_2^!(\Gamma_0(4f_\psi^2))$ such that $\Eis-G_\psi/\theta_\psi$ has no poles in the upper half-plane, so we may and will assume this from now on. By \Cref{propholproj}, we find that
\begin{gather}\label{eqpiholp}
\pihol^{reg}\left(\theta_\psi(p^{2a}\tau)\Eis_\psi(\tau)\right)=\theta_\psi(p^{2a}\tau)\Eis_\psi^+(\tau)+\sum_{r=1}^\infty \left(\sum_{\substack{(p^am)^2-n^2=r \\ m,n\geq 0}}\psi(mn)(p^am-n)\right)q^r\end{gather}
is a weakly holomorphic quasimodular form. This is congruent to a weakly holomorphic modular form $F\in M_2^!(\Gamma_0(4f_\psi^2p^{2a}))$ because the Eisenstein series $E_2$ is congruent to the holomorphic modular form $E_2(\tau)-NE_2(N\tau)\in M_2(\Gamma_0(N))$ for any $N>1$.

Since $(p^am)^2-n^2=(p^am-n)(p^am+n)$ factors, we can rewrite the inner sum in the second term in \eqref{eqpiholp} as a sum over small divisors of $r$,
\begin{gather}\label{eqmindiv}
\sum_{\substack{(p^am)^2-n^2=r \\ m,n\geq 0}}\chi(mn)(p^am-n)=\sum_{\substack{d\mid r,d\leq r/d \\ d+r/d\equiv 0\pmod{2p^a} \\ d-r/d\equiv 0\pmod 2}} \chi\left(\frac{(r/d)^2-d^2}{4p^a}\right)d,\end{gather}
as in the proof of \Cref{HMFGenFcn}. We mow apply the operator $U(p^b)$ to \eqref{eqpiholp}. This corresponds to replacing $r$ by $p^br$ in \eqref{eqmindiv}, and then the condition $d+p^br/d\equiv 0\pmod{2p^a}$ forces $d\equiv 0\pmod{p^{\min(a,b)}}$, so that we find that
\[\pihol^{reg}\left(\theta_\psi(p^{2a}\tau)\Eis_\psi(\tau)\right)|U(p^b)\equiv\left(\theta_\psi(p^{2a}\tau)\Eis_\psi^+(\tau)\right)|U(p^b)\equiv F|U(p^b)\pmod{p^{\min(a,b)}}.\]
By \cite[Lemma 6 and Lemma 7]{AL70}, $F|U(p^b)$ is modular for $\Gamma_0(4f_\psi^2p^m)$ for $m=\max(2a-b,1-2v_p(f_\psi))$, as stated in the Theorem. This completes the proof.
\end{proof}

\end{document}